%% file: excess.tex
\documentclass[a4paper,11pt]{amsart}

\input{preamble}

\input{macros}

\title{Virtual excess intersection theory\vspace{-2mm}}
\author{Adeel~A.~Khan\vspace{-1mm}}
\date{2021-03-17}

\makeatletter
\def\l@subsection{\@tocline{2}{0pt}{4pc}{6pc}{}}
\makeatother

\begin{document}

\begin{abstract}
We prove a K-theoretic excess intersection formula for derived Artin stacks.
When restricted to classical schemes, it gives a refinement and new proof of R.~Thomason's formula.
\vspace{-5mm}
\end{abstract}

\dedicatory{To my little brother, on the occasion of his birthday.}
\thanks{Author partially supported by SFB 1085 Higher Invariants, Universität Regensburg}

\maketitle

\renewcommand\contentsname{\vspace{-1cm}}
\tableofcontents

\parskip 0.2cm
\thispagestyle{empty}


\changelocaltocdepth{1}
\section*{Introduction}

  Suppose given a commutative square of \dAss
    \begin{equation*}
      \begin{tikzcd}
        \sX' \ar{r}{f'}\ar[swap]{d}{p}
          & \sY' \ar{d}{q}
        \\
        \sX \ar[swap]{r}{f}
          & \sY
      \end{tikzcd}
    \end{equation*}
  where $f$ and $f'$ are quasi-smooth closed immersions.
  We call this an \emph{excess intersection square} if the following conditions hold:
    \begin{remlist}
      \item
      The square is cartesian on underlying classical stacks, i.e., the canonical morphism $\sX' \to \sX \fibprod_\sY \sY'$ induces an isomorphism $\sX'_\cl \simeq (\sX \fibprod_\sY \sY')_\cl$.

      \item\label{item:surj}
      The canonical morphism of (shifted) relative cotangent complexes
        \begin{equation*}
          p^*\sL_{\sX/\sY}[-1] \to \sL_{\sX'/\sY'}[-1]
        \end{equation*}
      is surjective (on $\pi_0$).
      In particular, its fibre $\Delta$ is locally free of finite rank.
    \end{remlist}

  For a derived stack $\sX$, let $\K(\sX)$ denote the algebraic K-theory space of perfect complexes on $\sX$.
  In this note we prove the following theorem:

  \begin{thm}\label{thm:excess}
    For any excess intersection square as above, there is a canonical homotopy
      \begin{equation*}
        q^*f_*(-) \simeq f'_*(p^*(-) \cup e(\Delta))
      \end{equation*}
    of maps $\K(\sX) \to \K(\sY')$, where $e(\Delta) \in \K(\sX')$ is the Euler class of the excess sheaf.
  \end{thm}

  Applied to quotient stacks, this gives an equivariant excess intersection formula:

  \begin{cor}\label{cor:equiv}
    Let $G$ be a flat group algebraic space of finite presentation over an algebraic space $S$.
    For any excess intersection square of $G$-equivariant derived algebraic spaces
    \begin{equation*}
      \begin{tikzcd}
        X' \ar{r}{f'}\ar[swap]{d}{p}
          & Y' \ar{d}{q}
        \\
        X \ar{r}{f}
          & Y,
      \end{tikzcd}
    \end{equation*}
    where $f$ and $f'$ are regular closed immersions, there is a canonical homotopy
    \begin{equation*}
      q^*f_*(-) \simeq f'_*(p^*(-) \cup e(\Delta))
    \end{equation*}
    of maps $\K^G(X) \to \K^G(Y')$.
    Here $\Delta$ is the excess sheaf as above and $\K^G$ is the algebraic K-theory of $G$-equivariant perfect complexes.
  \end{cor}

  Derived stacks arise naturally in moduli theory, especially in curve counting theories such as Gromov--Witten theory and Donaldson--Thomas theory.
  The moduli problems arising in these theories are typically singular, but are nevertheless \emph{quasi-smooth} when regarded as derived stacks.
  As an example, \thmref{thm:excess} can be applied to derived moduli stacks of stable maps and allows a one-line K-theoretic proof of a conjecture of Cox, Katz and Lee \cite{CoxKatzLee} (cf. \cite[Cor.~2.2.6, Eq.~(4)]{Kern}) relating the genus zero Gromov--Witten invariants of a smooth projective variety with those of the zero locus of a section of a convex vector bundle.
  Via the virtual Grothendieck--Riemann--Roch theorem of \cite{KhanVirtual} one also recovers the Chow group level formula proven in \cite{KimKreschPantev}.

  The present paper, together with \cite{KhanKblow} and \cite{KhanGRRLect}, is part of a general program that sets up a K-theoretic formalism of intersection theory on derived schemes and stacks following \cite{SGA6}.
  In particular, the current result was applied in \cite{KhanGRRLect} to generalize the Grothendieck--Riemann--Roch theorem of \cite{SGA6} to quasi-smooth projective morphisms of derived schemes.
  If $X$ is a (possibly derived) scheme, then for any quasi-smooth closed immersion $i : Z \to X$, the class
    \begin{equation*}
      i_*(1) = [i_*(\sO_Z)] \in \K(X)
    \end{equation*}
  is the K-theoretic version of the cohomological virtual fundamental class $[Z]$ constructed in \cite[(3.21)]{KhanVirtual}.
  The GRR formula implies in particular that this class lives in the expected degree of the $\gamma$-filtration (determined by the relative virtual dimension).
  Thus one gets a map from the free abelian group on quasi-smooth closed subschemes (``derived cycles''), graded by virtual dimension, to the graded pieces of the $\gamma$-filtration on K-theory.

  Both \thmref{thm:excess} and \corref{cor:equiv} seem to be new even in the setting of classical algebraic geometry.
  A closed immersion of classical stacks is quasi-smooth if and only if it is regular (a.k.a. a local complete intersection) in the sense of \cite[Exp.~VII, Déf.~1.4]{SGA6} (see \cite[2.3.6]{KhanRydh}), so in that case the statement becomes:

  \begin{cor}
    For any cartesian square of Artin stacks
      \begin{equation*}
        \begin{tikzcd}
          \sX' \ar{r}{f'}\ar[swap]{d}{p}
            & \sY' \ar{d}{q}
          \\
          \sX \ar{r}{f}
            & \sY,
        \end{tikzcd}
      \end{equation*}
    where $f$ and $f'$ are regular closed immersions, there is a canonical homotopy
      \begin{equation*}
        q^*f_*(-) \simeq f'_*(p^*(-) \cup e(\Delta))
      \end{equation*}
    of maps $\K(\sX) \to \K(\sY')$, where $\Delta$ is the excess sheaf.
  \end{cor}

  For quasi-compact quasi-separated schemes, the excess intersection formula was proven by Thomason \cite[Thm.~3.1]{Thomason}.
  Even in that case our result is more precise in that the proof provides an explicit chain of homotopies between the two sides.

  For a homotopy cartesian\footnote{We remind the reader that a commutative square of \emph{classical} stacks is homotopy cartesian (in the \inftyCat of derived stacks) if and only if it is cartesian and Tor-independent.} square, the excess sheaf $\Delta$ vanishes, so the excess intersection formula reduces to the base change formula.
  More interesting are the following two special cases:

  \begin{cor}[Self-intersections]
    For any quasi-smooth closed immersion $f : \sX \hookrightarrow \sY$ of \dAss, there is a canonical homotopy
      \begin{equation*}
        f^*f_* \simeq (-) \cup e(\sN_{\sX/\sY})
      \end{equation*}
    of maps $\K(\sX) \to \K(\sX)$.
  \end{cor}

  This is the result of applying \thmref{thm:excess} to the self-intersection square
    \begin{equation*}
      \begin{tikzcd}
        \sX \ar[equals]{r}\ar[equals]{d}
          & \sX \ar{d}{f}
        \\
        \sX \ar{r}{f}
          & \sY,
      \end{tikzcd}
    \end{equation*}
  where the excess sheaf is the conormal sheaf $\sN_{\sX/\sY} = \sL_{\sX/\sY}[-1]$.
  A self-intersection formula for regular equivariant classical algebraic spaces was previously obtained by G.~Vezzosi and A.~Vistoli in \cite[Thm.~2.1]{VezzosiVistoli}\footnote{%
    Moreover, the refined statement we prove here, identifying an explicit homotopy between the two sides of the formula, is strong enough to correct the gap in the proof of \cite[Thm.~3.2]{VezzosiVistoli} that the authors point out in the erratum.
  }.
  Note that the equivariant version of this formula can be used to give a simple derivation of the virtual Atiyah--Bott formula for localizing to the fixed points of a torus action as in \cite[Sect.~3]{Qu} and \cite[Sect.~5]{CiocanFontanineKapranov}.

  Similarly, we get a generalization of the ``formule clef'' (key formula) of \cite[Exp.~VII]{SGA6}:

  \begin{cor}[Blow-ups]
    For any quasi-smooth closed immersion $f : \sX \hookrightarrow \sY$ of \dAss, consider the blow-up square
      \begin{equation*}
        \begin{tikzcd}
          \bP(\sN_{\sX/\sY}) \ar{r}{i}\ar[swap]{d}{p}
            & \Bl_{\sX}(\sY) \ar{d}{q}
          \\
          \sX \ar{r}{f}
            & \sY.
        \end{tikzcd}
      \end{equation*}
    Then there is a canonical homotopy
      \begin{equation*}
        q^*f_* \simeq i_*(p^*(-) \cup e(\Delta))
      \end{equation*}
    of maps $\K(\sX) \to \K(\Bl_{\sX}(\sY))$.
  \end{cor}

  The blow-up square is the universal excess square over $f : \sX \to \sY$ such that the upper morphism $i$ is a virtual Cartier divisor; see \cite{KhanRydh}.

  The proof of \thmref{thm:excess} is inspired by Fulton's proof of the Grothendieck--Riemann--Roch formula \cite[Chap.~15]{Fulton}, and uses a similar argument involving the deformation space $\Bl_{\sX\times\{\infty\}}(\sY\times\P^1)$ to reduce to the case where $f$ and $f'$ are zero sections of (projective completions of) vector bundles.
  Due to the failure of $\A^1$-homotopy invariance in algebraic K-theory (of singular spaces), the reduction is much more subtle than for the analogous result in motivic Borel--Moore homology (= higher Chow groups, under suitable hypotheses), which was proven in \cite[Prop.~3.15]{KhanVirtual}.

  The reader will find that the proof also goes through for any ``additive invariant'' in place of algebraic K-theory, i.e., for invariants of stable \inftyCats that satisfy additivity in the sense of Waldhausen.

\section{Derived symmetric powers}
\label{sec:sym}

  Let $\sX$ be a \dAs.
  We denote by $\Qcoh(\sX)$ the stable presentable \inftyCat of quasi-coherent $\sO_{\sX}$-modules, and by $\Qcoh(\sX)_{\ge 0}$ the full subcategory of connective objects.
  Recall that $\Qcoh(\sX)$ is the limit
    \begin{equation*}
      \lim_{u : X \to \sX} \Qcoh(X)
    \end{equation*}
  taken over the \inftyCat $\Lis(\sX)$ of smooth morphisms $u : X \to \sX$, where $X$ is an affine derived scheme.
  In other words, the sheaf $\sX \mapsto \Qcoh(\sX)$ is right Kan extended from affines.
  For an affine $X=\Spec(R)$, $\Qcoh(X)$ is equivalent to the stable \inftyCat of (nonconnective) modules over the \scr $R$, in the sense of \cite[Not.~25.2.1.1]{SAG}.
  Similarly for the connective and perfect subcategories $\Qcoh(-)_{\ge 0}$ and $\Perf(-)$, respectively.
  Recall that $\K(\sX)$ is defined as the algebraic K-theory space of the stable \inftyCat $\Perf(\sX)$.
  See \cite{KhanKblow} or \cite{KhanKstack} for more details on these definitions.
  We will in particular make use of the base change and projection formulas (see \cite[Rem.~1.9]{KhanKstack} or \cite[Cor.~3.4.2.2 (3), Rem.~3.4.2.6]{SAG}).
  
  By $\QcohAlg(\sX)$ we denote the presentable \inftyCat of quasi-coherent $\sO_{\sX}$-algebras.
  This admits a similar description as above, and for $\Spec(R)$ is equivalent to the \inftyCat of $R$-algebras.
  
  The \emph{derived symmetric algebra} functor
    \begin{equation*}
      \Sym^*_{\sO_\sX} : \Qcoh(\sX)_{\ge 0} \to \QcohAlg(\sX)
    \end{equation*}
  is left adjoint to the forgetful functor.
  For any morphism $f : \sX' \to \sX$ there are natural isomorphisms
    \begin{equation*}
      f^*(\Sym^*_{\sO_{\sX}}(\sF)) \simeq \Sym^*_{\sO_{\sX'}}(f^*\sF)
    \end{equation*}
  for every $\sF \in \Qcoh(\sX)_{\ge0}$ (since the forgetful functors commute with $f_*$).

  The derived symmetric algebra $\Sym^*_{\sO_\sX}(\sF)$ can be described in terms of the derived symmetric powers $\Sym^n_{\sO_\sX}(\sF)$.
  These are constructed in the affine case in \cite[Sect.~25.2.2]{SAG}, and extend to stacks by descent.
  (We warn the reader that, even for ordinary commutative rings, the derived symmetric powers are different from the classical ones on non-flat modules; see \cite[Prop.~25.2.3.4]{SAG} or \cite[Sect.~7]{QuillenHomology} for a classical introduction.)

  \begin{constr}
    Let $q : (\SCRMod)_{\ge0} \to \SCRing$ denote the cocartesian fibration associated to the presheaf of \inftyCats $R \mapsto (\Mod_R)_{\ge0}$; its objects are pairs $(R, M)$, where $R\in\SCRing$ is a \scr and $M\in(\Mod_R)_{\ge0}$ is a connective $R$-module.
    By \cite[Constr.~25.2.2.1]{SAG}, there is for each integer $n\ge0$ a functor $(\SCRMod)_{\ge0} \to (\SCRMod)_{\ge0}$ given informally by the assignment
      \begin{equation*}
        (R, M) \mapsto (R, \Sym^n_R(M)).
      \end{equation*}
    This functor preserves $q$-cocartesian morphisms \cite[Prop.~25.2.3.1]{SAG} and therefore induces functors
      \begin{equation*}
        \Sym^n_R : (\Mod_R)_{\ge 0} \to (\Mod_R)_{\ge 0},
      \end{equation*}
    which define a natural transformation of functors as $R$ varies.
    By right Kan extension, this extends to a natural transformation on the \inftyCat of \dAss.
    In other words for every \dAs $\sX$ we have functors
      \begin{equation*}
        \Sym^n_{\sO_\sX} : \Qcoh(\sX)_{\ge 0} \to \Qcoh(\sX)_{\ge 0},
      \end{equation*}
    which commute with $f^*$.
  \end{constr}

  \begin{lem}\label{lem:Sym powers}
    Let $\sX$ be a \dAs.
    Then for every connective quasi-coherent sheaf $\sE \in \Qcoh(\sX)_{\ge0}$, there is a canonical isomorphism
      \begin{equation*}
        \Sym^*_{\sO_\sX}(\sE)
          \simeq \bigoplus_{n\ge0} \Sym^n_{\sO_\sX}(\sE)
      \end{equation*}
    in $\Qcoh(\sX)$.
  \end{lem}

  \begin{proof}
    As $\sX$ varies, both $\Sym^*_{\sO_{\sX}}(-)$ and $\bigoplus_n \Sym^n_{\sO_{\sX}}(-)$ define natural transformations $\Qcoh(-)_{\ge0} \to \Qcoh(-)_{\ge0}$ of presheaves on the \inftyCat of \dAss.
    On the restrictions to affines, the two are canonically equivalent by \cite[Constr.~25.2.2.6]{SAG}.
    Since $\Qcoh(-)_{\ge0}$ is right Kan extended from affines, the claim follows.
  \end{proof}

  \begin{lem}\label{lem:Sym determinant}
    Let $\sX$ be a \dAs and $\sE' \to \sE \to \sE''$ a cofibre sequence of connective perfect complexes.
    Then there are canonical equivalences
      \begin{equation*}
        [\Sym^n_{\sO_{\sX}}(\sE)]
          \simeq \bigoplus_{i+j=n} [\Sym^i_{\sO_{\sX}}(\sE') \otimes_{\sO_{\sX}} \Sym^j_{\sO_{\sX}}(\sE'')]
      \end{equation*}
    for every $n\ge0$.
    If $\Sym^n_{\sO_{\sX}}(\sE) = 0$ for all sufficiently large $n \gg 0$, and similarly for $\sE'$ and $\sE''$, then also
      \begin{equation*}
        [\Sym^*_{\sO_{\sX}}(\sE)]
          \simeq [\Sym^*_{\sO_{\sX}}(\sE') \otimes_{\sO_{\sX}} \Sym^*_{\sO_{\sX}}(\sE'')]
      \end{equation*}
    in $\K(\sX)$.
  \end{lem}

  \begin{proof}
    For every such cofibre sequence and every integer $n\ge0$, there is a canonical filtration
      \begin{equation*}
        \Sym^n_{\sO_\sX}(\sE') = F^{0,n}
          \to F^{1,n}
          \to \cdots
          \to F^{n,n} = \Sym^n_{\sO_\sX}(\sE)
      \end{equation*}
    together with cofibre sequences
      \begin{equation*}
        F^{i-1,n} \to F^{i,n} \to \Sym^{n-i}_{\sO_\sX}(\sE') \otimes_{\sO_\sX} \Sym^{i}_{\sO_\sX}(\sE'')
      \end{equation*}
    for each $0 < i \le n$.
    This is constructed in \cite[Constr.~25.2.5.4]{SAG} for affines and the construction extends to stacks by a similar right Kan extension procedure.
    Since $\Sym^n_{\sO_{\sX}}(\sE)$ is perfect for all $n$ \cite[Prop.~25.2.5.3]{SAG} (and similarly for $\sE'$ and $\sE''$), these cofibre sequences give rise to the desired equivalences in $\K(\sX)$.
    The second claim follows from \lemref{lem:Sym powers}, since the assumption guarantees $\Sym^*_{\sO_{\sX}}(\sE)$ is also perfect (and similarly for $\sE'$ and $\sE''$).
  \end{proof}

  \begin{defn}
    The Euler class of a finite locally free sheaf $\sE$ is defined by
      \begin{equation*}
        e(\sE)
          = [\Sym^*_{\sO_\sX}(\sE[1])]
          = \sum_{n\ge 0} (-1)^n \cdot [\Sym^n_{\sO_\sX}(\sE[1])] \in \K(\sX).
      \end{equation*}
  \end{defn}

  There are canonical isomorphisms $\Sym^n_{\sO_\sX}(\sE[1]) \simeq \Lambda^n_{\sO_\sX}(\sE)[n]$ (cf. \cite[Prop.~25.2.4.2]{SAG}), where $\Lambda^n_{\sO_\sX}(-)$ denotes the derived exterior power, so this agrees with the usual definition of the Euler class (often denoted $\lambda_{-1}(\sE)$).

\section{Projective bundles}
\label{sec:vector bundle}

  Given a connective perfect complex $\sF$ on a \dAs $\sX$, there is an associated ``generalized vector bundle''
    \begin{equation*}
      \bV_\sX(\sF) = \Spec_\sX(\Sym^*_{\sO_\sX}(\sF)),
    \end{equation*}
  defined as the relative spectrum of its derived symmetric algebra.
  It is the moduli of cosections $\sF\to\sO_\sX$; that is, for any derived scheme $S$ over $\sX$, the space of $S$-points of $\bV_\sX(\sF)$ over $\sX$ is the space of $\sO_S$-linear homomorphisms $\sF|_S \to \sO_S$.
  
  This construction can exhibit some surprising behaviour.
  For example, if $\sE$ is a finite locally free sheaf, consider the morphism $p : \bV_\sX(\sE[1]) \to \sX$.
  This is a quasi-smooth closed immersion that fits in the homotopy cartesian square
    \begin{equation*}
      \begin{tikzcd}
        \bV_\sX(\sE[1]) \ar{r}{p}\ar[swap]{d}{p}
          & \sX \ar{d}{s}\\
        \sX \ar{r}{s}
          & \bV_\sX(\sE),
      \end{tikzcd}
    \end{equation*}
  where $s$ is the zero section.
  By the base change formula we obtain a canonical isomorphism
    \begin{equation*}
      s^*s_* \simeq p_*p^* \simeq (-) \otimes_{\sO_{\sX}} \Sym^*_{\sO_{\sX}}(\sE[1])
    \end{equation*}
  of functors $\Perf(\sX) \to \Perf(\sX)$, where the second isomorphism follows from $p_*(\sO) = \Sym^*_{\sO_\sX}(\sE[1])$ and the projection formula.
  We have just proven the following lemma in the special case where $t=s$:

  \begin{lem}\label{lem:i_*O zero locus}
    Let $\sX$ be a \dAs.
    Given a finite locally free sheaf $\sE$ and a cosection $t : \sE \to \sO_\sX$, we let $i : \sZ \to \sX$ denote its derived zero locus, so that there is a homotopy cartesian square
      \begin{equation*}
        \begin{tikzcd}
          \sZ \ar{r}{i}\ar[swap]{d}{i}
            & \sX \ar{d}{t}\\
          \sX \ar{r}{s}
            & \bV_\sX(\sE).
        \end{tikzcd}
      \end{equation*}
    Then there is an essentially unique homotopy
      \begin{equation*}
        i_*i^* \simeq (-) \cup e(\sE)
      \end{equation*}
    of maps $\K(\sX) \to \K(\sX)$.
  \end{lem}

  \begin{proof}
    The result of composing such a square with the open immersion $\bV_\sX(\sE) \to \bP(\sE\oplus\sO)$ is still homotopy cartesian (see \cite[Subsect.~3.1]{KhanKblow} for background on projective bundles in this setting).
    Let $\bar{s} : \sX \to \bP(\sE\oplus\sO)$ and $\bar{t} : \sX \to \bP(\sE\oplus\sO)$, respectively, denote the induced morphisms.
    By the base change formula we have
      \begin{equation*}
        i_*i^*
        \simeq \bar{t}^*\bar{s}_*
      \end{equation*}
    as functors $\Perf(\sX) \to \Perf(\sX)$.
    In the case $t=s$, we have as above $\bar{s}^*\bar{s}_* \simeq (-) \otimes_{\sO_\sX} \Sym^*_{\sO_\sX}(\sE[1])$, so in particular $\bar{s}^*\bar{s}_* \simeq (-) \cup e(\sE)$ in K-theory.
    Thus it will suffice to exhibit an (essentially unique) homotopy between the two maps $\K(\bP(\sE\oplus\sO)) \to \K(\sX)$ induced by $\bar{s}^*$ and $\bar{t}^*$.
    But from the projective bundle formula \cite[Cor.~3.4.1]{KhanKblow} it follows that there is an exact triangle
      \begin{equation*}
        \K(\P_\sX(\sE))
          \xrightarrow{\infty_*} \K(\P_\sX(\sE\oplus\sO))
          \xrightarrow{\bar{u}^*} \K(\sX),
      \end{equation*}
    for \emph{any} $\bar{u} : \sX \to \bP(\sE\oplus\sO)$ induced by a section $u : \sX \to \bV_\sX(\sE)$.
    In other words, both maps in question are the cofibre of the same map $\infty_*$, whence the desired homotopy.
  \end{proof}

  Let $\pi : \P(\sE\oplus\sO) \to \sX$ denote the projection.
  We have on $\P(\sE\oplus\sO)$ the canonical exact triangle of locally free sheaves
    \begin{equation*}
      \sQ \to \pi^*(\sE)\oplus\sO \to \sO(1).
    \end{equation*}
  Recall that the zero section $\bar{s} : \sX\to\P(\sE\oplus\sO)$ can be written as the derived zero locus of the canonical cosection
    \begin{equation*}
      \sQ \to \pi^*(\sE)\oplus\sO \xrightarrow{\mrm{pr}_2} \sO.
    \end{equation*}
  Thus we get:

  \begin{cor}\label{cor:sbar_*}
    There is a canonical homotopy
      \begin{equation*}
        \bar{s}_*(-) \simeq e(\sQ) \cup \pi^*(-)
      \end{equation*}
    of maps $\K(\sX) \to \K(\P(\sE\oplus\sO))$.
  \end{cor}

  \begin{proof}
    By \lemref{lem:i_*O zero locus}, $\bar{s}_*(\sO) \simeq e(\sQ)$.
    By the projection formula, $\bar{s}_*(-) \simeq \bar{s}_*(\sO) \cup \pi^*(-) \simeq e(\sQ) \cup \pi^*(-)$.
  \end{proof}

  We are now ready to prove a special case of \thmref{thm:excess}.
  Let $p : \sX' \to \sX$ be a morphism of \dAss.
  Let $\sE$ and $\sE'$ be finite locally free sheaves on $\sX$ and $\sX'$, respectively, together with a surjection
    \begin{equation*}
      p^*(\sE) \twoheadrightarrow \sE'
    \end{equation*}
  whose fibre we denote $\Delta$.
  This induces an excess intersection square
    \begin{equation*}
      \begin{tikzcd}
        \sX' \ar{r}{\bar{s}'}\ar{d}{p}
          & \bP(\sE'\oplus\sO) \ar{d}{q}
        \\
        \sX \ar{r}{\bar{s}}
          & \bP(\sE\oplus\sO),
      \end{tikzcd}
    \end{equation*}
  where $\bar{s}$ and $\bar{s}'$ are the zero sections.

  \begin{claim}\label{claim:projective completions}
    The excess intersection formula
      \begin{equation*}
        q^* \bar{s}_*
          \simeq \bar{s}'_* (p^*(-) \cup e(\Delta))
      \end{equation*}
    holds for the above square.
  \end{claim}

  \begin{proof}
    Let $\pi : \bP(\sE\oplus\sO) \to \sX$ and $\pi' : \bP(\sE'\oplus\sO) \to \sX'$ denote the respective projections.
    Let $\sQ$ and $\sQ'$ denote the respective universal hyperplane sheaves on $\P(\sE\oplus\sO)$ and $\P(\sE'\oplus\sO)$.
    The surjection $p^*(\sE) \to \sE'$ gives rise to a canonical morphism $q^*\sQ \to \sQ'$, whose fibre is $(\pi')^*(\Delta)$.
    Thus \lemref{lem:Sym determinant} provides a canonical homotopy
      \begin{equation*}
        e(q^*\sQ) \simeq e((\pi')^*\Delta) \cup e(\sQ')
      \end{equation*}
    in $\K(\P(\sE'\oplus\sO))$.
    Now two applications of \corref{cor:sbar_*} give:
      \begin{align*}
        q^* \bar{s}_*
          &\simeq q^*(\pi^*(-) \cup e(\sQ))\\
          &\simeq (\pi')^* p^*(-) \cup e(q^*\sQ)\\
          &\simeq (\pi')^* p^*(-) \cup e((\pi')^*\Delta) \cup e(\sQ')\\
          &\simeq (\pi')^* (p^*(-) \cup e(\Delta)) \cup e(\sQ')\\
          &\simeq \bar{s}'_* (p^*(-) \cup e(\Delta)),
      \end{align*}
    as desired.
  \end{proof}

\section{Deformation space}
\label{sec:deformation}

  Let $f : \sX \to \sY$ be a quasi-smooth closed immersion of \dAss.
  Write $M$ for the blow-up $\Bl_{\sX\times\{\infty\}}(\sY\times\P^1)$ as in \cite{KhanRydh}.
  It fits in a commutative diagram
    \begin{equation*}
      \begin{tikzcd}
        \sX \ar{r}{s_0}\ar[swap]{d}{f}
          & \sX \times \P^1 \ar{d}{\hat{f}}
          & \sX \ar[swap]{l}{s_\infty}\ar{d}{f_\infty}
        \\
        \sY \ar{r}{\sigma_0}\ar[swap]{d}{\pi_0}
          & M \ar{d}{\widehat{\pi}}
          & \P(\sN_{\sX/\sY}\oplus\sO) \ar[swap]{l}{\sigma_\infty}\ar{d}{\pi_\infty}
        \\
        \{0\} \ar{r}
          & \P^1
          & \{\infty\} \ar{l}
      \end{tikzcd}
    \end{equation*}
  The two left-hand squares and upper right-hand square are homotopy cartesian.
  The morphism $\hat{f}$ is
    \begin{equation*}
      \sX\times\P^1 = \Bl_{\sX\times\{\infty\}}(\sX\times\P^1) \to \Bl_{\sX\times\{\infty\}}(\sY\times\P^1),
    \end{equation*}
  induced by $f\times\id : \sX\times\P^1\to\sY\times\P^1$, and the morphism $f_\infty$ is the zero section.

  Denote by $M_\infty := M \fibprodR_{\P^1} \{\infty\}$ the special fibre, and by $i_\infty : M_\infty \to M$ the inclusion.
  Then we have a canonical homotopy
    \begin{equation}\label{eq:linear equivalence}
      (\sigma_0)_*(\sigma_0)^* \simeq (i_\infty)_*(i_\infty)^*
    \end{equation}
  of maps $\K(M)\to\K(M)$.
  Indeed, we have $0_*(\sO) \simeq \infty_*(\sO)$ in $\K(\P^1)$, so by the base change formula there is a canonical identification
    \begin{align*}
      (\sigma_0)_*(\sO)
        \simeq (\sigma_0)_*(\pi_0)^*(\sO)
        &\simeq \widehat{\pi}^*0_*(\sO)
      \\
        &\simeq \widehat{\pi}^*\infty_*(\sO)
        \simeq (i_\infty)_*(\pi_\infty)^*(\sO)
        \simeq (i_\infty)_*(\sO)
    \end{align*}
  in $\K(M)$.  Thus the claim follows from the projection formula.

  The fibre $M_\infty$ fits in a homotopy cartesian and cocartesian square
    \begin{equation*}
      \begin{tikzcd}
        \P(\sN_{\sX/\sY}) \ar{r}\ar{d}
          & \Bl_{\sX}(\sY) \ar{d}
        \\
        \P(\sN_{\sX/\sY}\oplus\sO) \ar{r}
          & M_\infty.
      \end{tikzcd}
    \end{equation*}
  That is, $M_\infty$ is the sum of the two virtual Cartier divisors $\P(\sN_{\sX/\sY}\oplus\sO)$ and $\Bl_{\sX}(\sY)$ on $M$.
  We denote by $i_\infty : M_\infty \to M$ the inclusion, and by $b : \Bl_{\sX}(\sY) \to M$, and $c : \P(\sN_{\sX/\sY}) \to M$ the composites with $i_\infty$.
  We have a canonical homotopy
    \begin{equation}\label{eq:divisors sum}
      (i_\infty)_*(i_\infty)^*
        \simeq (\sigma_\infty)_*(\sigma_\infty)^* + b_*b^* - c_*c^*
    \end{equation}
  of maps $\K(M) \to \K(M)$, by the following lemma.

  \begin{lem}\label{lem:divisors sum}
    Let $D \hookrightarrow \sX$ and $D' \hookrightarrow \sX$ be virtual Cartier divisors on a \dAs $\sX$.
    Denote by $D \cap D' = D \fibprodR_\sX D'$ their intersection and by $$D + D' = D \fibcoprod_{D\cap D'} D'$$ their sum.
    Then we have a canonical homotopy
      \begin{equation*}
        (i_{D+ D'})_*(i_{D+ D'})^*
          \simeq (i_D)_*(i_D)^* + (i_{D'})_*(i_{D'})^* - (i_{D\cap D'})_*(i_{D\cap D'})^*
      \end{equation*}
    of maps $\K(\sX) \to \K(\sX)$.
  \end{lem}

  \begin{proof}
    By definition of $D+D'$ we have
      \begin{equation*}
        (i_{D+D'})_*(\sO_{D+ D'})
          \simeq (i_D)_*(\sO_D) \fibprod_{(i_{D\cap D'})_*(\sO_{D\cap D'})} (i_{D'})_*(\sO_{D'})
      \end{equation*}
    in $\Perf(\sX)$.
    This induces in $\K(\sX)$ a canonical homotopy
      \begin{equation*}
        (i_{D+ D'})_*(\sO_{D+ D'})
          \simeq (i_D)_*(\sO_D) + (i_{D'})_*(\sO_{D'}) - (i_{D\cap D'})_*(\sO_{D\cap D'}).
      \end{equation*}
    We conclude using the projection formula.
  \end{proof}

  Since the intersection
    \begin{equation*}
      \Bl_{\sX}(\sY) \fibprod_{M} (\sX\times\P^1)
        = \Bl_{\sX}(\sY) \fibprod_{M_\infty} \sX
        = \P(\sN_{\sX/\sY}) \fibprod_{\P(\sN_{\sX/\sY\oplus\sO})} \sX
    \end{equation*}
  is empty, we have $b^*\hat{f}_* = 0$ and $c^*\hat{f}_* = 0$ by the base change formula.
  Thus \eqref{eq:linear equivalence} and \eqref{eq:divisors sum} induce the homotopy
    \begin{equation}\label{eq:linear equivalence 2}
      (\sigma_0)_*(\sigma_0)^* \hat{f}_*
        \simeq (i_\infty)_*(i_\infty)^* \hat{f}_*
        \simeq (\sigma_\infty)_*(\sigma_\infty)^* \hat{f}_*
    \end{equation}
  of maps $\K(\sX\times\P^1) \to \K(M)$.

\section{Proof}
\label{sec:proof}

  Consider an excess intersection square of the form
    \begin{equation*}
      \begin{tikzcd}
        \sX' \ar{r}{f'}\ar[swap]{d}{p}
          & \sY' \ar{d}{q}
        \\
        \sX \ar[swap]{r}{f}
          & \sY.
      \end{tikzcd}
    \end{equation*}
  We keep the notation of the previous section, so $M = \Bl_{\sX\times\{\infty\}}(\sY\times\P^1)$, etc.
  We consider all the same constructions for $f' : \sX' \to \sY'$, with notation decorated by primes: $M' = \Bl_{\sX'\times\{\infty\}}(\sY'\times\P^1)$, and so on.
  We have morphisms of excess intersection squares
    \begin{equation*}
      \begin{tikzcd}
        \sX' \ar{r}{f'}\ar[swap]{d}{p}
          & \sY' \ar{d}{q}
        \\
        \sX \ar[swap]{r}{f}
          & \sY.
      \end{tikzcd}
      \quad\hookrightarrow\quad
      \begin{tikzcd}
        \sX'\times\P^1\ar{r}{\hat{f'}}\ar{d}{\hat{p}}
          & M' \ar{d}{\hat{q}}
        \\
        \sX\times\P^1\ar{r}{\hat{f}}
          & M
      \end{tikzcd}
      \quad\hookleftarrow\quad
      \begin{tikzcd}
        \sX' \ar{r}{f'_\infty}\ar[swap]{d}{p}
          & \P(\sN_{\sX'/\sY'}\oplus\sO) \ar{d}{q_\infty}
        \\
        \sX \ar[swap]{r}{f_\infty}
          & \P(\sN_{\sX/\sY}\oplus\sO).
      \end{tikzcd}
    \end{equation*}
  That the middle square is an excess intersection square is clear from the observation that the surjectivity condition~\ref{item:surj} can be checked on the fibres of $\sX' \times \P^1$.
  
  Consider the canonical morphisms $r : \sX \times \P^1 \to \sX$, $\rho : M \to \sY$ (as well as their primed versions), retractions of $s_0 : \sX \to \sX \times \P^1$ and $\sigma_0 : \sY \to M$, respectively.
  Using $r\circ s_0 = \id$ and the base change formula, we have
    \begin{align*}
      q^*f_*
        &\simeq q^*f_*(s_0)^*r^* \\
        &\simeq q^*(\sigma_0)^*\hat{f}_*r^* \\
        &\simeq (\sigma'_0)^* \hat{q}^* \hat{f}_* r^*
    \end{align*}
  and similarly
    \begin{align*}
      (q_\infty)^* (f_\infty)_*
        &\simeq (q_\infty)^* (f_\infty)_*(s_\infty)^* r^*\\
        &\simeq (q_\infty)^*(\sigma_\infty)^* \hat{f}_*r^*\\
        &\simeq (\sigma'_\infty)^*\hat{q}^*\hat{f}_*r^*.
    \end{align*}
  Thus \eqref{eq:linear equivalence 2} induces an equivalence
    \begin{equation*}
      (\sigma'_0)_* q^*f_*
        \simeq (\sigma'_0)_*(\sigma'_0)^* \hat{q}^* \hat{f}_* r^*
        \simeq (\sigma'_\infty)_*(\sigma'_\infty)^* \hat{q}^* \hat{f}_* r^*
        \simeq (\sigma'_\infty)_* (q_\infty)^* (f_\infty)_*.
    \end{equation*}
  Applying $\rho'_*$ gives
    \begin{equation*}
      q^*f_* \simeq \rho'_* (\sigma'_\infty)_* (q_\infty)^* (f_\infty)_*
    \end{equation*}
  since $\rho' \circ \sigma'_0 = \id$.
  Finally, \claimref{claim:projective completions} yields the desired equivalence
    \begin{equation*}
      q^*f_*
        \simeq \rho'_* (\sigma'_\infty)_* (f'_\infty)_* (p^*(-) \cup e(\Delta))
        \simeq f'_* (p^*(-) \cup e(\Delta)).
    \end{equation*}

\changelocaltocdepth{2}


\bibliographystyle{halphanum}

\bigskip \noindent {\href{mailto:khan@ihes.fr}{khan@ihes.fr}} \medskip

\noindent IHES\\
\noindent 35 route de Chartres\\
\noindent Bures-sur-Yvette\\
\noindent 91440 France\\

\noindent Institute of Mathematics\\
\noindent Academia Sinica\\
\noindent Taipei\\
\noindent 10617 Taiwan

\end{document}

%% file: preamble.tex
\usepackage{hyperref}
\usepackage{xcolor}
\hypersetup{
    colorlinks,
    linkcolor={red!50!black},
    citecolor={blue!50!black},
    urlcolor={blue!80!black}
}

\usepackage{microtype}
\usepackage{MnSymbol}

\makeatletter
\newcommand{\eqnum}{\refstepcounter{equation}\textup{\tagform@{\theequation}}}
\makeatother

\makeatletter \@addtoreset{equation}{section} \makeatother
\renewcommand{\theequation}{\thesection.\arabic{equation}}
\newtheorem{defn}[equation]{Definition}
\newtheorem{thm}[equation]{Theorem}
\newtheorem*{thm*}{Theorem}
\newtheorem{lem}[equation]{Lemma}
\newtheorem{cor}[equation]{Corollary}

\newtheorem{claim}[equation]{Claim}
\newtheorem*{defthm*}{Definition/Theorem}

\theoremstyle{definition}

\newtheorem{constr}[equation]{Construction}

\newtheorem*{exam*}{Example}

\newcommand\arXiv[1]{\href{http://arxiv.org/abs/#1}{arXiv:#1}}

\usepackage[utf8]{inputenc}

\usepackage{xspace}

\newcommand{\changelocaltocdepth}[1]{%
  \addtocontents{toc}{\protect\setcounter{tocdepth}{#1}}%
  \setcounter{tocdepth}{#1}}

\newcommand{\nc}{\newcommand}
\nc{\renc}{\renewcommand}
\nc{\ssec}{\subsection}
\nc{\sssec}{\subsubsection}
\nc{\on}{\operatorname}
\nc{\term}[1]{#1\xspace}

\usepackage{tikz}
\usetikzlibrary{matrix}
\usepackage{tikz-cd}

\tikzset{
  commutative diagrams/.cd,
  arrow style=tikz,
  diagrams={>=latex}}
\makeatletter
\tikzset{
  column sep/.code=\def\pgfmatrixcolumnsep{\pgf@matrix@xscale*(#1)},
  row sep/.code   =\def\pgfmatrixrowsep{\pgf@matrix@yscale*(#1)},
  matrix xscale/.code=%
    \pgfmathsetmacro\pgf@matrix@xscale{\pgf@matrix@xscale*(#1)},
  matrix yscale/.code=%
    \pgfmathsetmacro\pgf@matrix@yscale{\pgf@matrix@yscale*(#1)},
  matrix scale/.style={/tikz/matrix xscale={#1},/tikz/matrix yscale={#1}}}
\def\pgf@matrix@xscale{1}
\def\pgf@matrix@yscale{1}
\makeatother

\usepackage{enumitem}
\setlist[enumerate,1]{label={(\alph*)},itemsep=\parskip,leftmargin=0pt}
\newlist{thmlist}{enumerate}{1}
\setlist[thmlist,1]{label={\em(\roman*)},ref={\upshape{(\roman*)}},itemsep=\parskip,leftmargin=0pt}     
\newlist{remlist}{enumerate}{1}
\setlist[remlist,1]{label={(\roman*)},itemsep=\parskip,leftmargin=0pt}
\newlist{inlinelist}{enumerate*}{1}
\setlist[inlinelist,1]{label={(\alph*)}}

\nc{\sA}{\ensuremath{\mathcal{A}}\xspace}
\nc{\sB}{\ensuremath{\mathcal{B}}\xspace}
\nc{\sC}{\ensuremath{\mathcal{C}}\xspace}
\nc{\sD}{\ensuremath{\mathcal{D}}\xspace}
\nc{\sE}{\ensuremath{\mathcal{E}}\xspace}
\nc{\sF}{\ensuremath{\mathcal{F}}\xspace}
\nc{\sG}{\ensuremath{\mathcal{G}}\xspace}
\nc{\sH}{\ensuremath{\mathcal{H}}\xspace}
\nc{\sI}{\ensuremath{\mathcal{I}}\xspace}
\nc{\sJ}{\ensuremath{\mathcal{J}}\xspace}
\nc{\sK}{\ensuremath{\mathcal{K}}\xspace}
\nc{\sL}{\ensuremath{\mathcal{L}}\xspace}
\nc{\sM}{\ensuremath{\mathcal{M}}\xspace}
\nc{\sN}{\ensuremath{\mathcal{N}}\xspace}
\nc{\sO}{\ensuremath{\mathcal{O}}\xspace}
\nc{\sP}{\ensuremath{\mathcal{P}}\xspace}
\nc{\sQ}{\ensuremath{\mathcal{Q}}\xspace}
\nc{\sR}{\ensuremath{\mathcal{R}}\xspace}
\nc{\sS}{\ensuremath{\mathcal{S}}\xspace}
\nc{\sT}{\ensuremath{\mathcal{T}}\xspace}
\nc{\sU}{\ensuremath{\mathcal{U}}\xspace}
\nc{\sV}{\ensuremath{\mathcal{V}}\xspace}
\nc{\sW}{\ensuremath{\mathcal{W}}\xspace}
\nc{\sX}{\ensuremath{\mathcal{X}}\xspace}
\nc{\sY}{\ensuremath{\mathcal{Y}}\xspace}
\nc{\sZ}{\ensuremath{\mathcal{Z}}\xspace}

\nc{\bA}{\ensuremath{\mathbf{A}}\xspace}
\nc{\bB}{\ensuremath{\mathbf{B}}\xspace}
\nc{\bC}{\ensuremath{\mathbf{C}}\xspace}
\nc{\bD}{\ensuremath{\mathbf{D}}\xspace}
\nc{\bE}{\ensuremath{\mathbf{E}}\xspace}
\nc{\bF}{\ensuremath{\mathbf{F}}\xspace}
\nc{\bG}{\ensuremath{\mathbf{G}}\xspace}
\nc{\bH}{\ensuremath{\mathbf{H}}\xspace}
\nc{\bI}{\ensuremath{\mathbf{I}}\xspace}
\nc{\bJ}{\ensuremath{\mathbf{J}}\xspace}
\nc{\bK}{\ensuremath{\mathbf{K}}\xspace}
\nc{\bL}{\ensuremath{\mathbf{L}}\xspace}
\nc{\bM}{\ensuremath{\mathbf{M}}\xspace}
\nc{\bN}{\ensuremath{\mathbf{N}}\xspace}
\nc{\bO}{\ensuremath{\mathbf{O}}\xspace}
\nc{\bP}{\ensuremath{\mathbf{P}}\xspace}
\nc{\bQ}{\ensuremath{\mathbf{Q}}\xspace}
\nc{\bR}{\ensuremath{\mathbf{R}}\xspace}
\nc{\bS}{\ensuremath{\mathbf{S}}\xspace}
\nc{\bT}{\ensuremath{\mathbf{T}}\xspace}
\nc{\bU}{\ensuremath{\mathbf{U}}\xspace}
\nc{\bV}{\ensuremath{\mathbf{V}}\xspace}
\nc{\bW}{\ensuremath{\mathbf{W}}\xspace}
\nc{\bX}{\ensuremath{\mathbf{X}}\xspace}
\nc{\bY}{\ensuremath{\mathbf{Y}}\xspace}
\nc{\bZ}{\ensuremath{\mathbf{Z}}\xspace}

\nc{\bbA}{\ensuremath{\mathbb{A}}\xspace}
\nc{\bbB}{\ensuremath{\mathbb{B}}\xspace}
\nc{\bbC}{\ensuremath{\mathbb{C}}\xspace}
\nc{\bbD}{\ensuremath{\mathbb{D}}\xspace}
\nc{\bbE}{\ensuremath{\mathbb{E}}\xspace}
\nc{\bbF}{\ensuremath{\mathbb{F}}\xspace}
\nc{\bbG}{\ensuremath{\mathbb{G}}\xspace}
\nc{\bbH}{\ensuremath{\mathbb{H}}\xspace}
\nc{\bbI}{\ensuremath{\mathbb{I}}\xspace}
\nc{\bbJ}{\ensuremath{\mathbb{J}}\xspace}
\nc{\bbK}{\ensuremath{\mathbb{K}}\xspace}
\nc{\bbL}{\ensuremath{\mathbb{L}}\xspace}
\nc{\bbM}{\ensuremath{\mathbb{M}}\xspace}
\nc{\bbN}{\ensuremath{\mathbb{N}}\xspace}
\nc{\bbO}{\ensuremath{\mathbb{O}}\xspace}
\nc{\bbP}{\ensuremath{\mathbb{P}}\xspace}
\nc{\bbQ}{\ensuremath{\mathbb{Q}}\xspace}
\nc{\bbR}{\ensuremath{\mathbb{R}}\xspace}
\nc{\bbS}{\ensuremath{\mathbb{S}}\xspace}
\nc{\bbT}{\ensuremath{\mathbb{T}}\xspace}
\nc{\bbU}{\ensuremath{\mathbb{U}}\xspace}
\nc{\bbV}{\ensuremath{\mathbb{V}}\xspace}
\nc{\bbW}{\ensuremath{\mathbb{W}}\xspace}
\nc{\bbX}{\ensuremath{\mathbb{X}}\xspace}
\nc{\bbY}{\ensuremath{\mathbb{Y}}\xspace}
\nc{\bbZ}{\ensuremath{\mathbb{Z}}\xspace}

\DeclareMathSymbol{A}{\mathalpha}{operators}{`A}
\DeclareMathSymbol{B}{\mathalpha}{operators}{`B}
\DeclareMathSymbol{C}{\mathalpha}{operators}{`C}
\DeclareMathSymbol{D}{\mathalpha}{operators}{`D}
\DeclareMathSymbol{E}{\mathalpha}{operators}{`E}
\DeclareMathSymbol{F}{\mathalpha}{operators}{`F}
\DeclareMathSymbol{G}{\mathalpha}{operators}{`G}
\DeclareMathSymbol{H}{\mathalpha}{operators}{`H}
\DeclareMathSymbol{I}{\mathalpha}{operators}{`I}
\DeclareMathSymbol{J}{\mathalpha}{operators}{`J}
\DeclareMathSymbol{K}{\mathalpha}{operators}{`K}
\DeclareMathSymbol{L}{\mathalpha}{operators}{`L}
\DeclareMathSymbol{M}{\mathalpha}{operators}{`M}
\DeclareMathSymbol{N}{\mathalpha}{operators}{`N}
\DeclareMathSymbol{O}{\mathalpha}{operators}{`O}
\DeclareMathSymbol{P}{\mathalpha}{operators}{`P}
\DeclareMathSymbol{Q}{\mathalpha}{operators}{`Q}
\DeclareMathSymbol{R}{\mathalpha}{operators}{`R}
\DeclareMathSymbol{S}{\mathalpha}{operators}{`S}
\DeclareMathSymbol{T}{\mathalpha}{operators}{`T}
\DeclareMathSymbol{U}{\mathalpha}{operators}{`U}
\DeclareMathSymbol{V}{\mathalpha}{operators}{`V}
\DeclareMathSymbol{W}{\mathalpha}{operators}{`W}
\DeclareMathSymbol{X}{\mathalpha}{operators}{`X}
\DeclareMathSymbol{Y}{\mathalpha}{operators}{`Y}
\DeclareMathSymbol{Z}{\mathalpha}{operators}{`Z}

\nc{\mrm}[1]{\ensuremath{\mathrm{#1}}\xspace}
\nc{\mit}[1]{\ensuremath{\mathit{#1}}\xspace}
\nc{\mbf}[1]{\ensuremath{\mathbf{#1}}\xspace}
\nc{\mcal}[1]{\ensuremath{\mathcal{#1}}\xspace}
\nc{\msc}[1]{\ensuremath{\mathscr{#1}}\xspace}

\renc{\ge}{\geqslant}
\renc{\le}{\leqslant}

\nc{\id}{\mathrm{id}}

\DeclareMathOperator{\Hom}{\on{Hom}}
\nc{\uHom}{\underline{\smash{\Hom}}}

\DeclareMathOperator{\End}{\on{End}}
\DeclareMathOperator{\Sym}{\on{Sym}}
\nc{\uEnd}{\underline{\smash{\End}}}

\nc{\colim}{\varinjlim}
\renc{\lim}{\varprojlim}
\nc{\Cofib}{\on{Cofib}}
\nc{\Fib}{\on{Fib}}
\nc{\initial}{\varnothing}
\nc{\op}{\mathrm{op}}

\DeclareMathOperator*{\fibprod}{\times}
\DeclareMathOperator*{\fibcoprod}{\operatorname{\sqcup}}

\renc{\setminus}{\smallsetminus}

\usepackage{mathtools}
\newcommand{\thmref}[1]{Theorem~\ref{#1}}

\newcommand{\lemref}[1]{Lemma~\ref{#1}}

\newcommand{\corref}[1]{Corollary~\ref{#1}}
\newcommand{\claimref}[1]{Claim~\ref{#1}}

\renewcommand{\eqref}[1]{(\ref{#1})}

%% file: macros.tex
\nc{\SCRing}{\mrm{SCRing}}
\nc{\SCRMod}{\on{SCRMod}}
\nc{\Mod}{\on{Mod}}
\nc{\A}{\bA}
\renc{\P}{\bP}
\nc{\Spec}{\on{Spec}}
\nc{\Qcoh}{\on{Qcoh}}
\nc{\QcohAlg}{\on{QcohAlg}}
\nc{\Perf}{\on{Perf}}
\nc{\cl}{{\mrm{cl}}}
\nc{\Bl}{\on{Bl}}
\nc{\fibprodR}{\fibprod^\bR}
\nc{\rightrightrightarrows}
  {\mathrel{\substack{\textstyle\rightarrow\\[-0.6ex]
   \textstyle\rightarrow \\[-0.6ex]
   \textstyle\rightarrow}}}
\nc{\rightrightrightrightarrows}
  {\mathrel{\substack{\textstyle\rightarrow\\[-0.6ex]
   \textstyle\rightarrow \\[-0.6ex]
   \textstyle\rightarrow \\[-0.6ex]
   \textstyle\rightarrow}}}
\nc{\Lis}{\mrm{Lis}}
\nc{\K}{\on{K}}

\nc{\scr}{\term{simplicial commutative ring}}
\nc{\scrs}{\term{simplicial commutative rings}}

\nc{\inftyCat}{\term{$\infty$-category}}
\nc{\inftyCats}{\term{$\infty$-categories}}

\nc{\dAs}{\term{derived Artin stack}}
\nc{\dAss}{\term{derived Artin stacks}}